\documentclass[reqno,11pt]{amsart}
\usepackage{amscd,amssymb, amsmath}
\usepackage{graphicx}
\usepackage[frenchb, russian, english]{babel}
\usepackage[utf8]{inputenc}

\newtheorem{theorem}{Theorem}[section]
\newtheorem{lemma}[theorem]{Lemma}
\newtheorem{corollary}[theorem]{Corrolary}
\theoremstyle{definition}
\newtheorem{definition}[theorem]{Definition}

\newtheorem{proposition}[theorem]{Claim}
\newtheorem{assumption}[theorem]{Assumption}

\newtheorem{example}[theorem]{Example}

\newtheorem{conjecture}[theorem]{Hypothesis}

\newtheorem{note}[theorem]{Note}
\newtheorem{nthm}{Theorem}[section]

\newtheorem{ncnj}{Conjecture}[section]

\DeclareMathOperator{\Hom}{Hom}
\DeclareMathOperator{\Sing}{Sing}
\DeclareMathOperator{\Spec}{Spec}

\DeclareMathOperator{\Res}{Res}

\DeclareMathOperator{\Fin}{Fin}
\DeclareMathOperator{\Ob}{Ob}
\DeclareMathOperator{\Or}{Or}
\DeclareMathOperator{\Aut}{Aut}
\DeclareMathOperator{\tSet}{tSet}
\DeclareMathOperator{\ssSet}{ssSet}
\DeclareMathOperator{\red}{red}

\DeclareMathOperator{\ad}{ad}
\DeclareMathOperator{\im}{Im}
\DeclareMathOperator{\const}{const}
\newcommand{\lExt}{\text{\textit{Ext}}}
\newcommand{\lRHom}{\text{\textit{RHom}}}

\begin{document}
\author{Lev Soukhanov \\ \tiny{National Research University Higher School of Economics}}
\title{A Surface Degeneration With Non-Collapsible Contractible Dual Complex}
\pagestyle{plain}
\maketitle
\begin{abstract}
We construct a smoothing of a particular normal crossing surface, which dual complex is a duncehat, thus obtaining an example of a surface degeneration with non-collapsible, contractible dual complex. The general fiber of this family has $h^{1,0} = h^{2,0} = 0$ and $h^{1,1}=9$, we conjecture that it is in the deformation class of the Barlow surface. This text is a complete rewrite of author's previous work \cite{LS}.
\end{abstract}

\section{Introduction}

This paper is a cleaned up and improved version of author's preprint \cite{LS}. It features more straightforward version of the main theorem and various improvements and clarifications.

\begin{subsection}{Main question} Let use denote by $D \subset \mathbb{C}$ the unit holomorphic disk and consider an analytic family $\pi: \tilde{X} \rightarrow D$ of complex surfaces over a disk, with smooth total space, with central fiber $X = \pi^{-1}(0)$ being a normal crossing surface. Denote by $X^{\varepsilon} = \pi^{-1}(\varepsilon)$ the nearby fiber, by $X_1, \cdots, X_n$ the irreducible components of $X$, and by $\hat{X}_1, \cdots, \hat{X}_n$ their normalizations.

Suppose $\hat{X}_i$ are rational surfaces, and normalizations of components of the singular locus are rational curves (we call such degeneration \textbf{geometrically maximal}). Then, the \textbf{dual complex} of the special fiber $X$ is intimately related to the topology of the general fiber - in particular, if the dual complex has trivial rational cohomology, then the general fiber has $h^{1,0}(X^{\varepsilon})=h^{2,0}(X^{\varepsilon})=0$. By the results of Kollar, Xu and de Fernex \cite{dFKX}, Theorem 4, if the general fiber was rational, then the dual complex needs to be contractible and even \textbf{collapsible}, which poses an interesting question of building (from smoothing techniques) examples of degenerations with dual complex being contractible but not collapsible. It would give (possibly new) surfaces with $h^{1,0} = h^{2,0} = 0$.
\end{subsection}

\begin{subsection}{Differences with earlier versions} As this text is an overhaul of the previous version of the paper, it might be worth stating the differences with the previous version. The main difference is the different proof of surjectivity of the map $O$ - instead of dealing with the asymptotics we just calculate it explicitly enough. We also give independent proofs of all the results we use for the smoothing of normal crossing spaces instead of referring to the literature, which, sadly, mostly written in the assumptions of simple normal crossing. Since the time we writing previous text, we have also managed to prove that our example is actually simply-connected, which we only suspected at that time.
\end{subsection}

\begin{subsection}{Purpose and organization of the paper}
This paper presents an example (first explicit, to our knowledge) of a geometrically maximal degeneration with non-collapsible contractible dual complex (topological duncehat), which answers affirmatively to the question of \cite{dFKX} on whether such degenerations exist. The generic fiber of this degeneration turns out to be some simply-connected surface of general type with $h^{1,1} = 9$ - same as Barlow's surface \cite{B}. We conjecture that it is in a deformation class of Barlow's surface. We would like to also point out that there are constructions which produce simply-connected surfaces with $p_g = 0$ by smoothing of isolated $\mathbb{Q}$-Gorenstein singularities \cite{LYP}. It would be extremely interesting to find out whether these constructions are related to the present work - in particular, what are the dual complexes of SNC-resolutions of these families.

Our initial motivation for this construction was an attempt to understand the extent to which the two-dimensional complexes play the role similar to those of graphs in the enumerative geometry of surfaces. This understanding is still yet to be achieved.

Another motivation is the fact that this construction could, in principle, provide us with the new examples of surfaces with $h^{1,0}=h^{2,0}=0$.

We prefer to work in the assumptions which are slightly weaker than most of the works in the field - i.e., in the literature the degeneration is commonly assumed to be simple normal crossing, while we prefer to work with just normal crossing. Due to this, and for readers and our own convenience we tried to keep the exposition self-contained when possible. We would like to point out that while in principle all constructions can be done in simple normal crossing case the combinatorics of the problem tend to become very messy while doing so.
\newline

The text consists of the following next parts:

~

\noindent \textbf{Notation} - we state the problem precisely, introduce our assumptions and fix the terminology and notation for triangulated spaces and dual complexes.

~

\noindent \textbf{Smoothing} - we present our view on the deformation theory of the normal crossings (which largely follows Friedman, but we explain the relation to the logarithmic deformation theory) and reduce the smoothing problem to the calculation of the fiber of a certain obstruction mapping. The main result of this part is the following observation, known in the slightly different form already to Friedman:

~
    
\noindent For the normal crossing surface $S$ which satisfies our geometric maximality assumption denote by $M_S$ the moduli space of locally trivial deformations of $S$, by $J_S$ the space of line bundles on the singular locus of $S$ (it can be naturally identified for all locally trivial deformations of $S$), by $O: M_S \longrightarrow J_S$ the map which sends the normal crossing surface to its first tangent cohomology $T^1(S)$, considered as a line bundle on its singular locus. Assume also that locally trivial deformations of $S$ are unobstructed.
\begin{nthm}[Corollary \ref{cor:moduli-smooth}]Suppose that the fiber $O^{-1}(\mathcal{O}_{\Sing(S)})$ is smooth as a subscheme of $M_S$ in a point corresponding to $S$. Then, $S$ is smoothable.
\end{nthm}
It is clear that if the fiber is not smooth, but reduced, we can find the locally trivial deformation which is smoothable. We also prove the slight improvement of the result above - the case when the fiber is not reduced but has an expected dimension (Th. \ref{thm:cool-smooth}) This improvement is not needed in what follows but might turn out to be useful later. \\~\\~
\textbf{Example} - we introduce our degeneration, the special fiber $X$.\\~\\~
\textbf{Obstruction mapping} - we calculate the obstruction explicitly, thus proving \begin{nthm}[Corollary \ref{cor:epicwin}] There exists an analytic family $\tilde{X}$ with a smooth total space, with geometrically maximal degeneration and dual complex homeomorphic to the duncehat. \end{nthm}
\noindent
\textbf{Properties of the smoothing} - we show that the generic fibers of $\tilde{X}$ are algebraic, of general type and minimal; we calculate the fundamental group and show that it vanishes;
\newline 

\noindent \textbf{Questions} - we state and motivate few questions and conjectures based on our observations. The conjectures (one very simple and other maybe a bit too brave) are: \begin{ncnj}The generic fiber of our degeneration $X^{\varepsilon}$ is in a deformation class of Barlow's surface.\end{ncnj}
\begin{ncnj}Every simply-connected surface with $h^{1,0}=h^{2,0}=0$ admits a geometrically maximal degeneration with non-collapsible contractible dual complex.\end{ncnj}
    
\end{subsection}

\begin{subsection}{Acknowledgement}
I thank Konstantin Loginov and Vadim Vologodsky for their crucial impact on my understanding of the subject; Denis Teryoshkin, Dmitry Sustretov, Andrei Losev and Dmitry Korb for stimulating discussions; Mikhail Kapranov for his suggestion on the connection of this work with the possible combinatorial construction of the phantom categories.
\end{subsection}

\section{Notation}

\subsection{Normal crossing.} \hfill\\

We fix the terminology. Recall that the singular point is called \textbf{normal crossing} if it is locally (in etale or analytic topology) equivalent to a union of a collection of coordinate hyperplanes in $\mathbb{C}^n$.

\begin{definition}The variety with normal crossing singularities is called \textbf{normal crossing} variety.\end{definition}

This notion admits a few different versions, which we would like to distinguish between.

\begin{definition}
Suppose $X$ is a NC-variety such that all its irreducible components are smooth. Then, it is called \textbf{simple normal crossing} (SNC).

Suppose, in addition, that the intersection of any collection of its irreducible components is connected. Then, we will call it \textbf{strictly simple normal crossing} (SSNC).
\end{definition}

\subsection{Triangulated spaces.}\hfill\\

There are also few different but similar notions of a triangulated space.

\begin{definition}
The \textbf{semi-simplicial set} $S$ is a contravariant functor $S$ from the category of non-empty totally ordered finite sets with injective order-preserving morphisms $\Delta_+$ to the category of sets.

Explicitly, it yields a set $S_0 = S(\{0\})$ of vertices, set $S_1 = S(\{01\})$ of edges, set $S_2 = S(\{012\})$ of $2$-simplices e.t.c. together with natural face maps $\partial_k: S_n \rightarrow S_{n-1}, 0 \leq k \leq n$ corresponding to deletion of $k$-th vertex.

We denote the category of semi-simplicial sets as $\ssSet$.
\end{definition}

\begin{note} $1$-dimensional semi-simplicial set is the same as ordered graph (possibly with loops and double edges).
\end{note}

This notion, however, contains additional information which we do not actually need - every simplex in the semi-simplicial set has a chosen ordering of vertices. There is the following, slightly more general, notion which forgets this data.

\begin{definition}
The \textbf{triangulated set} $T$ is a contravariant functor $T$ from the category of finite sets with injective morphisms $\Fin$ to the category of sets with the following additional property: for any $F \in \Ob(\Fin)$ action of $\Aut(F)$ on $T(F)$ is free. We denote $T(n) = T_n$, as above.

The category of triangulated sets is denoted as $\tSet$.
\end{definition}

\begin{note}
$T_n$ should be thought of as the set of $n$-dimensional simplices of $T$ with chosen ordering of vertices.
\end{note}

\begin{example}
There is a functor $p: \ssSet \rightarrow \tSet$, defined as: $p(S)(X) = S(X) \times \Or(X)$, where $\Or(X)$ denotes the set of orderings of $X$, and the face maps are natural.
\end{example}

While every $1$-dimensional triangulated set comes from the semi-simplicial set, in dimension $2$ it is already false.

\begin{note}
Every semi-simplicial set (and every triangulated set) has a geometric realization functor. We will use the word ''triangulated space'' for the pair (semi-simplicial/triangulated set, its realization) and hope it won't cause any confusion.
\end{note}

The main triangulated space we want to use throughout our work, however, does come from a semi-simplicial set:

\begin{definition}\label{def:dunce}
The \textbf{topological duncehat} is a $2$-dimensional semi-simplicial set with one vertex, one edge and one triangle and obvious structure maps.
\end{definition}

\begin{note}
We can also define functor $q: \tSet \rightarrow \ssSet$ as follows.

For a triangulated set $T$ denote the set of facets without orientation $T_n^{\red} = T_n / \Aut(\{01...n\})$.

Now, $T^{\red}_n$ and $T^{\red}_{m}, m < n$ admit a correspondence, for $\alpha \in T^{\red}_{m}, \beta \in T^{\red}_n$ we say that $\alpha < \beta$ iff they have preimages in $\tilde{\beta} \in T_n$ and $\tilde{\alpha} \in T_{m}$ such that $\tilde{\alpha}$ is a face of $\tilde{\beta}$. Now, we define:

$q(T)_n = \{\text{ordered collections of the form }\alpha_0 < \alpha_1 < ... \alpha_n\text{ with }\alpha_i \in T^{\red}_{k_i}, k_0 < ... < k_n\}$. Such collections are called ''flags'', and composition of the functors $p$ and $q$ in either order corresponds to the barycentric subdivision of the corresponding triangulated spaces.
\end{note}

We also give two possible restrictions on the definition above.

\begin{definition}
The triangulated space will be called \textbf{simple} if any simplex has no coinciding faces (of any codimension). The triangulated space will be called \textbf{strictly simple} if any two simplices either have an empty intersection or have a maximal common face. 
\end{definition}

\begin{note}
Strictly simple triangulated space with global ordering on the set of vertices is the same as simplicial complex.
\end{note}

\begin{note}
$1$-dimensional triangulated space is a graph. Being simple for it means that the corresponding graph has no loops. Strictly simple means that it also has no double edges.
\end{note}

\subsection{Dual complexes.}
\begin{definition} For every SNC-variety $X = \underset{i=1}{\overset{n}{\bigcup}} X_i$ define its dual complex $\Delta_X$ by the following semi-simplicial set: \[(\Delta_X)_k = \underset{1 \leq i_0 < ... < i_k \leq n}{\coprod} \pi_0(X_{i_0}\cap X_{i_1} \cap ... \cap X_{i_k})\] with face maps coming from natural inclusions of $X_{i_0}\cap X_{i_1} \cap ... \cap X_{i_k} \subset X_{i_0}\cap X_{i_1} \cap ... X_{i_{s-1}} \cap X_{i_{s+1}}... \cap X_{i_k}$
\end{definition}

\begin{note}
$\Delta_X$ is strictly simple iff $X$ is strictly SNC and simple iff $X$ is SNC. 
\end{note}

We can also define the dual complex for some (but not all!) normal crossing varieties, however, it won't in general will be a semi-simplicial set, but a triangulated set.

\begin{definition}
Consider a normal crossing variety $X$ together with its natural stratification $Z_n \subset Z_{n-1} \subset ... \subset Z_{1} = \Sing(X) \subset Z_0 = X$. Denote by $\mathcal{F}_k$ the set of irreducible components of $Z_k$. Consider any $F \in \mathcal{F}_k$, $\nu_F: \hat{F} \rightarrow F$ - its normalization (it is smooth). For any point of $x \in \hat{F}$, denote by $B_x(\hat{F})$ the set of irreducible components of the analytic germs of $X$, containing $(\nu_F)_{*}(T_x(\hat{F}))$ (we will call it the set of ''branches'' in $F$). It forms a local system over $\hat{F}$.

We will say that $X$ satisfies \textbf{no branch switching} assumption if this set is a trivial local system for any $F \in \mathcal{F}_k$ for any $k$. In that case we will denote this (globally defined) set $B(\hat{F})$, emphasizing its independence on the point $x$.
\end{definition}

\begin{example}
The simplest example when this assumption fails is the following: consider $X$ to be some surface which contains an embedded elliptic curve $E$, denote by $\sigma: E \rightarrow E$ any free involution of $E$. Then, $X/(s \sim \sigma (s))$ does not satisfy the no branch switching assumption.
\end{example}

\begin{note}
No branch switching assumption holds automatically if every $F$ is simply connected. We are mostly interested in the case when $X$ is a surface and all components and all intersections are rational, so for our application it is indeed the case.
\end{note}

\begin{definition}
For a normal crossing variety $X$ satisfying no branch switching assumption the \textbf{dual complex} is a triangulated set $\Delta_X$ defined as:
\[(\Delta_X)_k = \{(F, \alpha)|F\in\mathcal{F}_k, \alpha \in  \Or(B(\hat{F}))\}\]
 and the face maps are defined as follows:
 
 for any pair $F \in \mathcal{F}_k, F' \in \mathcal{F}_{k-1}$ such that $F \subset F'$ the ordering of the branches in $F$ induces the ordering of branches in $F'$ naturally (for the reason that any branch of $F$ but one can be analytically continued to the branch of $F'$). That defines the boundary operation: if $(F, \alpha)$ and $(F', \alpha')$ are such that $F \subset F'$ and the orderings are induced (and the element $k \in \alpha$ is the one which is not promoted to $\alpha'$) we say that $\partial_k (F, \alpha) = (F', \alpha')$
\end{definition}

\subsection{Maximal vs geometrically maximal degeneration}\hfill\\

Consider the family $\pi: \tilde{X} \longrightarrow D$ of complex Kahler varieties over a disk, with smooth total space and NC central fiber. In the theory of Calabi-Yau degenerations there is a maximality condition \cite{D}.

\begin{definition}[Deligne]
The degeneration of Calabi-Yau manifolds of dimension $n$ is called \textbf{maximal} if an operator of monodromy acting on the middle cohomology $H^n(X)$ has a Jordan block of length $n$.
\end{definition}

This condition is conjectured to give the dual complexes homeomorphic to a (homological) sphere. It is indeed true in the case of K3 surfaces, according to Kulikov \cite{Ku} and Freedman \cite{Fr}.

In the case of surfaces of general type with $h^{1,0}=h^{2,0}=0$ there is clearly no hope of characterizing the maximality of degeneration in terms of the limiting Hodge structure - there is no variation of Hodge structure anywhere at all, as the Hodge structure of the generic fiber is already of Hodge-Tate type. We, however, give a condition which coincides with the maximality in case of K3 surfaces. We do not know its precise relation to the notion of maximality given by Deligne in Calabi-Yau case.

\begin{definition}
The degeneration of the family of varieties is called \textbf{geometrically maximal} if the normalizations of all irreducible components of $Z_i$ are rational.
\end{definition}

\begin{note}If the dimension of the fiber is $> 2$, it is not clear whether we should maybe use ''rationally connected'' instead of ''rational'' in this definition. We are mostly concerned about the dimension $2$ case, so it doesn't really matter.\end{note}

We would also like to point out that this definition is a bit ad hoc - it is completely not obvious how to prove that some surface (even admitting a lot of deformations) has such a degeneration.

\subsection{Cohomology of the dual complex and limiting Hodge structure}
Consider, the analytic family of algebraic varieties $\pi: \tilde{X} \longrightarrow D$ over a disk, with smooth total space and normal crossing central fiber. We denote as $\pi^{-1}(\varepsilon) = X^{\varepsilon}$ the fiber over some generic nearby point $\varepsilon \neq 0$. Assume, in addition, that the degeneration is geometrically maximal. Consider $\Delta_X$ - the dual complex of the central fiber .

\begin{lemma}
$H^*(\Delta_X, \mathbb{C}) = 0$ implies $h^{*,0}(X^{\varepsilon}) = 0$
\end{lemma}

This lemma (for the SNC case) can be deduced from the results of Schmid and Steenbrink \cite{St} on the limiting mixed Hodge structure, and improved to NC case due to the invariance of the homotopy type of the dual complex under birational changes (see \cite{dFKX}). We, however, give a very simple and straightforward argument:

\begin{proof} Denote by $X = \underset{a \in \pi_0 (\hat{X})}{\coprod} X_a$ the decomposition of the central fiber into irreducible components. Denote their normalizations as $\nu_a: \hat{X}_a \rightarrow X_a$.

For the stratum $Z_k$ denote its irreducible components $X_{\mu}$, where $\mu \in \Delta_X$, their normalizations $\nu_{\mu}: \hat{X}_{\mu} \rightarrow X_{\mu}$.

Then, we have the following resolution of the sheaf $\mathcal{O}$ on $X^0$:

\[
0 \rightarrow \mathcal{O}(X) \rightarrow \bigoplus_{\mu \in (\Delta_X)_0} \pi_* \mathcal{O}(\hat{X}_{\mu}) \otimes L_{\mu} \rightarrow \] \begin{equation} \rightarrow
 \bigoplus_{\mu \in (\Delta_X)_1} \pi_* \mathcal{O}(\hat{X}_{\mu}) \otimes L_{\mu} \rightarrow \bigoplus_{\mu \in (\Delta_X)_2} \pi_* \mathcal{O}(\hat{X}_{\mu}) \otimes L_{\mu} \rightarrow \cdots
\end{equation}

Here, $L_{\mu}$ is the $1$-dimensional orientation space of the simplex $\mu$. The maps are induced by the natural inclusions of the components $\hat{X}_{\mu} \rightarrow \hat{X}_{\mu'}$ and are taken with the sign depending on the chosen orientation of corresponding simplices.

Now, the components of this resolution are acyclic (being the direct images of $\mathcal{O}$ on some rational varieties via normalization morphism, which has no higher direct images), and, hence, it calculates the cohomology of $\mathcal{O}(X)$. Now, the complex
\[0 \rightarrow \bigoplus_{a \in (\Delta_X)_0} H^0(X, \pi_* \mathcal{O}(\hat{X}_{\mu})) \otimes(L_\mu)  \rightarrow \bigoplus_{\mu \in (\Delta_X)_1} H^0(X, \pi_* \mathcal{O}(\hat{X}_{\mu})) \otimes(L_{\mu}) \rightarrow \] \[ \bigoplus_{(\mu) \in (\Delta_X)_2} H^0(X, \pi_* \mathcal{O}(\hat{X}_{\mu})) \otimes L_{\mu} \rightarrow \cdots\]

is clearly identified with the combinatorial cochain complex of $\Delta_X$ because all $H^0(X^0, *)$ factors are isomorphic to $\mathbb{C}$ canonically, which implies the desired result.
\end{proof}

\section{Smoothing}

In this part, we present our point of view on the smoothing theory of normal crossing surfaces. It largely follows Friedman's original work on K3 surfaces \cite{Fr}.

Suppose we are given a normal crossing surface $X$. Then, there is the following question: is it possible to construct a family over a disk
$\tilde{X} \rightarrow D)$ with smooth total space, no multiple fibers and the central fiber isomorphic to $X$? There is an obstruction to it:

Consider the sheaf of first tangent cohomology $T_{X}^1$, it might be thought of as the sheaf of local deformations of the scheme. For the normal crossing surface, it forms a line bundle over the singular locus $\Sing(X)$.

The obstruction is as follows: $T_X^1$ must be isomorphic to $\mathcal{O}_{\Sing(X)}$.

$T_X^1$ of any normal crossing variety admits the very explicit description, which we reproduce, following Kawamata and Namikawa \cite{KN}.

\begin{definition}Let us denote by $R_k$ the coordinate cross in $\mathbb{C}^n$ defined by the equation $x_1 \cdots x_k = 0$. The \textbf{log chart} $U$ on $X$ is an (analytic) open subset $U$ together with the holomorphic embedding $\phi: U \longrightarrow \mathbb{C}^n$, such that $\phi(U)$ is identified with an open neighborhood of the $0$ in $R_k$. For the chart $U$ we denote $u_i = \phi^{*}x_i$. We will refer to $u_i$ for $i \in \{k+1, ..., n\}$ as ''free coordinates'' \end{definition}

\begin{definition}
The \textbf{log atlas} is the collection of log charts $U^1, ..., U^s$ together for the following structure for the pair of $U^i$ and $U^j$:

Choose the permutation $\sigma$ such that any irreducible component of $U^i \cap U^j$ such that it has index $s$ on $U^i$ has an index $\sigma(s)$ on $U^j$ - i.e., vanishing loci of $u^i_s$ and $u^j_{\sigma(s)}$ coincide.

Then, one should have the collection of functions $z^{ij}_1, ..., z^{ij}_n$ such that $u^j_{\sigma(s)} = z^{ij}_s u^i_s$. This choice is unique on $U^i \cap U^j \cap \Sing(X)$ up to the permutation of free coordinates.

This functions should satisfy the following condition: $z^{ij}_1 \cdots z^{ij}_n = 1$
\end{definition}

\begin{definition}
Suppose now $X$ was a central fiber of the family $\pi: \tilde{X} \rightarrow D$, and suppose $z$ is a standard coordinate on $D$. Then, the collection of open subsets $U_i$ of $\tilde{X}$ endowed with the maps $\phi_i: U_i \rightarrow \mathbb{C}^n$ such that $\phi_i^*x_1 \cdots \phi_i^*x_n = \pi^*z$ is called the \textbf{extended log atlas}.
\end{definition}

\begin{note}
It is easy to see that the extended log atlas on $\tilde{X}$ always exists, moreover, it is clear that the restriction of the extended log atlas on $X$ gives a log atlas. Hence, the existence of the log atlas on $X$ is the necessary condition for smoothing.
\end{note}

\begin{definition}[Kawamata-Namikawa]\label{def:log-kawa-nami}
Two log atlases are called \textbf{equivalent} if their union is a log atlas. The class of equivalences of log atlases is called the \textbf{log structure} on $X$.
\end{definition}

\begin{note}
In modern terminology this structure is called the log structure of semistable type. 
\end{note}

\begin{definition}\label{def:tx1-log}
Now, choose the collection of log charts, not necessary satisfying the log atlas condition. Then, we define the line bundle $T_X^1$ on $\Sing(X)$ by the gluing functions $\varphi^{ij} = (z^{ij}_1 \cdots z^{ij}_n)^{-1}$.
\end{definition}

The log structure is the same as the trivializing section of this bundle.

\begin{definition}
The variety with trivial $T_X^1$ is called $d$-semistable.
\end{definition}

Now, let us recall the relation of this definition of $T_X^1$ with the more standard ones. Exposition here closely follows Friedman \cite{Fr}.

Suppose $X$ is a normal crossing variety, locally (in analytic topology) embedded in a smooth variety $V$. Then, we have the following short exact conormal sequence, which is a locally free resolution of the sheaf of Kahler differentials on $X$: \begin{equation}\label{eqn:conormal-sequence} 0 \rightarrow I_X / I_X^2 \overset{e}{\rightarrow} \Omega^1_{V|X} \overset{r}{\rightarrow} \Omega^1_X \rightarrow 0\end{equation}
The dual sequence fails to be exact, and $T_X^1$ measures this failure: \begin{equation}\label{eqn:normal-sequence}0 \rightarrow T_X^0 \overset{r^*}{\rightarrow} T_{V|X} \overset{e^*}{\rightarrow} N_X \overset{\delta}{\rightarrow} \end{equation}
\[ \hspace*{-2cm} \rightarrow T_X^1 \rightarrow 0\]
The invariant version of this description is:\begin{equation}\label{eqn:tx1-invariant}T_X^1 = \lExt^1(\Omega_X, \mathcal{O}_X)\end{equation}

Now, the definition by the normal sequence from eq.\ref{eqn:normal-sequence} bears resemblance with our explicit description. Let us show it explicitly for the convenience.

\begin{lemma}\label{lem:tx1-equiv}The definition \ref{def:tx1-log} and definition by normal sequence (eq.\ref{eqn:normal-sequence}) are equivalent.\end{lemma}

\begin{proof}
Suppose we are given, as in the def. \ref{def:tx1-log}, the collection of local log charts. Consider, at first, one of the charts, $U$, together with its structure functions $u_1, ... u_n$, the structure embedding \[\phi = (u_1, ..., u_n): U \rightarrow \mathbb{C}^n = V\] Let its image be the open neighborhood of the zero in the coordinate cross $R_k$. The explicit generator of $I_{\phi(U)}/I_{\phi(U)}^2$ is the defining equation of $X_k$, the function $x_1 x_2 \cdots x_k$. As an element of $\Omega_{V|X}$ it is represented by the differential form \[\alpha = x_2 \cdots x_k dx_1 + ... + x_1 \cdots x_{k-1} dx_k\]

Now, we need to calculate explictly the map \[e^*: T_{V|X} \longrightarrow N_X\]

We will use an element $\alpha^* \in N_X$ as the trivializing section of the normal bundle. Then, for a tangent field $v \in T_{V|X}$, the following holds: \[e^*(v) = \alpha(v) \alpha^*\]

It is, thus, not too hard to calculate the image of this map: \[\im(e^*) = \langle x_2 \cdots x_k; x_1 x_3 \cdots x_k; ... ; x_1 \cdots x_{k-1}\rangle \alpha^*\]

That means that the cokernel of this map is supported on $\Sing(R_k)$ and is a line bundle on it, generated by the image of $\alpha^*$. It now remains to calculate the gluing map on the intersection of two charts.

To simplify the calculation, let us add to our collection of log charts all the intersections between them. The gluing map, so, w.l.o.g. can be calculated for the pair of charts $U' \subset U$. Note that for the specific choice of structure functions - structure functions on $U'$ being just the restrictions of the structure functions on $U$, the gluing is trivially $1$. So we can further assume w.l.o.g. that we are actually given one local chart $U$ with two different sets of structure functions, $(u_1, ..., u_n)$, $(u_1', ..., u_n')$, and we need to calculate the resulting relation of the elements $(\alpha^*)'/(\alpha^*)|_{\Sing(U)}$. Again, we can assume (because it is true up to the permutation) that zero locus of $u_i$ and $u_i'$ coincides for all $i$. That means that $z_i = u_i' / u_i$ is a well defined function on $\overline{U \setminus \{u(p)=0\}}$. Let us promote it (non-uniquely) to the holomorphic function $z_i$ on $U$, satisfying $u_i' = z_i u_i$.

As the next step, consider the structure map $\phi = (u_1, ..., u_n): U \rightarrow V$. Let us, further, find such invertible holomorphic functions $\tilde{z}_1, ..., \tilde{z}_n$ that $\phi^* (\tilde{z}_i) = z_i$. Then, define the new coordinate chart \[x_i' = \tilde{z}_i x_i\] Note that $\phi$ in this new coordinate chart is actually the structure map for the collection $(u_1', ..., u_n')$: \[\phi^*(x_i') = u_i'\]

That allows us to compare \[\alpha = d(x_1\cdots x_n))\] and \[\alpha = d(x_1'\cdots x_n') = \tilde{z}_1 \cdots \tilde{z}_n d(x_1 \cdots x_n) + x_1 \cdots x_n d(\tilde{z}_1 \cdots \tilde{z}_n)\]

The second term vanishes on $\im(\phi)$, which shows that \[\alpha' = z_1 \cdots z_n \alpha\]
which concludes the proof.
\end{proof}

Now, let us proceed with the general theory of smoothing of normal crossing varieties. The exposition, once again, closely follows Friedman \cite{Fr}, the general results allowing this description are by Tyurina and Palamodov \cite{P}.

\begin{definition}\label{def:tx-global}For any proper algebraic variety $X$ we have the tangent complex \[T_X = \lRHom(\Omega_X, \mathcal{O}_X)\]\end{definition} its global first cohomology is the tangent space of local deformations, and the formal scheme of infinitesimal deformations of $X$ being the fiber of (non-linear formal) Kuranishi map: \begin{equation}K: H^1(T_X) \longrightarrow H^2(T_X)\end{equation}

In normal crossings case, $T_X$ has only 0-th and 1-st cohomology as the complex of sheaves. That gives the distinguished triangle \begin{equation} \label{eq:tx-dist-triangle}T_X^0 \rightarrow T_X \rightarrow T_X^1[1]\end{equation} leading to the following exact sequence:
\[0 \rightarrow H^0(T_X^0) \rightarrow H^0(T_X) \rightarrow 0 \rightarrow \]
\begin{equation}\label{eq:tx-exact-seq}\rightarrow H^1(T_X^0) \rightarrow H^1(T_X) \rightarrow H^0(T_X^1) \rightarrow\end{equation}
\[\rightarrow H^2(T_X^0) \rightarrow H^2(T_X^1) \rightarrow H^1(T_X) \rightarrow ...\]

Suppose now that the dimension of $X$ is $2$. Then, all further terms vanish automatically by dimension. We now restrict our interest to the smoothing formal deformations - i.e., such families \[\tilde{X} \rightarrow \Spec(\mathbb{C}[[t]])\] that the total space is smooth, and the central fiber is reduced. It implies the statement equivalent to Kulikov's triple point condition:
\begin{lemma}[reformulation of Kulikov's triple point condition] \label{lem:Kul-c1}
$ $ \newline $c_1(T_X^1)=0$ on every irreducible component of $\Sing(X)$.
\end{lemma}
\begin{proof}Actually, $T_X^1$ should be a trivial bundle since such deformation requires the existence of log structure (see def.\ref{def:log-kawa-nami})\end{proof}

Let us now describe the geometric significance of the sequence from eq.\ref{eq:tx-exact-seq}.

\begin{note}\label{note:geom-sense}
The $T_X^0$ is just the sheaf of derivations of $X$. So, the fact that $H^0(T_X^0)$ and $H(T_X)$ are identified is clear from the standpoint of the deformation theory - any infinitesimal automorphism of $X$ comes from the vector field on $X$, just as in the smooth case.

Now, the cohomologies of $T_X^0$ actually control \textbf{locally trivial} deformations - that is, deformations which preserve the structure of singularities. We will see it more explicitly later, or one could refer to \cite{Fr}. The locally trivial deformations embed into the tangent space of all deformations, $H^1(T_X)$. Any nontrivial deformation gives the section of $H^0(T_X^1)$. Now, assuming Kulikov's topological condition (see lemma \ref{lem:Kul-c1}) we get that either $h^0(T_X^1) = 1$ if $T_X^1$ is trivial, or $h^0(T_X^1) = 0$ if it is nontrivial, yet trivial topologically. In the latter case, any deformation is locally trivial (which improves already known fact that there are no smoothings in this case). In the former case, the mapping from $H^1(T_X)$ to $H^0(T_X^1)$ could still be zero. This, however, can be avoided in case $H^2(T_X^0) = 0$ (a natural assumption corresponding to the idea that $X$, as a gluing of components, is not ''overdetermined'', so it has non-obstructed locally trivial deformations). In that case, we will have $h^1(T_X) = h^1(T_X^0) + 1$, and the Kuranishi map $K$ is a formal map of the form \[K:H^1(T_X) \rightarrow H^1(T_X^1)\] Later, the very explicit description of this map will be given. Motivated by this, we from now on assume the following:

\end{note}

\begin{assumption}\label{as:local-unobstructed}$H^2(T_X^0) = 0$\end{assumption}

In particular, this assumption implies that the locally trivial deformations of $X$ are unobstructed.

Now, we would like to describe the Kuranishi mapping. Our main source of knowledge here is the discussion in the Chapter 4 in \cite{Fr}.

According to \cite{P} $T_X$ forms a sheaf of dg-algebras. The $[\cdot,\cdot]$ descends on cohomology, which gives the quadratic mapping
\[K_2: H^1(T_X) \rightarrow H^2(T_X)\]
\[K_2(\alpha) = [\alpha, \alpha]\]

By the general deformation theory it is the second term of the Kuranishi mapping $K = \frac{1}{2}K_2 + \frac{1}{6}K_3 + \frac{1}{24}K_4 + ...$

Now, in our case the cohomologies of $T_X$ have a filtration: \[0 \rightarrow H^*(T_X^0) \rightarrow H^*(T_X) \rightarrow H^*(T_X^1) \rightarrow 0 \]

This filtration is in agreement with the commutator - that is, $H^*(T_X^0 \oplus T_X^1[1])$ is an associated graded algebra of $H^*(T_X)$.

Moreover, the commutator in this subalgebra has the following explicit description:

for $a, b \in H^*(T_X^0):$ \[[a,b] \text{ comes from the commutator of vector fields}\]

for $a \in H^*(T_X^0), b \in H^*(T_X^1[1]):$ \[ [a,b] = L_a(b)\] comes from the action of vector fields on all natural objects.

for $a, b \in H^*(T_X^1[1])$: \[[a,b] = 0\]  for obvious grading reasons.

Now, suppose $X$ had a log structure. Then, denote by $\xi \in H^0(T_X) = H^1(T_X^1[1])$ the corresponding trivializing section of the line bundle $T_X^1$ (this is the same as $\alpha^*$, but we use it in the cohomological calculations, and $\alpha^*$ in the calculations with local coordinates). 

\begin{definition}\label{def:tx-log}
The sheaf $T_X(\log) \subset T_X$ is a subsheaf of all vector fields $v$ which preserve log structure: i.e. \[[v, \xi] = 0\]

It is seen by direct local inspection that this sheaf is locally free (and $T_X$ is not).
\end{definition}

It is readily seen (again, verified locally) that $T_X(\log)$ admits the following resolution:

\begin{equation}\label{eq:tx-log-resolved}0 \rightarrow T_X(\log) \overset{i}{\rightarrow} T_X^0 \overset{\ad_{\xi}}{\rightarrow} T_X^1 \rightarrow 0\end{equation}

\begin{lemma} $\ad_\xi$ is actually a morphism of \textbf{coherent} sheaves i.e. it is linear w.r.t. multiplication by functions from the ground ring. \end{lemma}

Statement of this lemma follows from proposition 4.3 of \cite{Fr}.

\begin{proof}
In the local chart $U$ with structure functions $u_1, ..., u_n$ we act by a vector field $v$. Its infinitesimal action is \[u_i \mapsto u_i + \epsilon \partial_v(u_i)\]

The change functions to this new set of structure functions, so, are \[z_i = 1 + \epsilon \frac{\partial_v(u_i)}{u_i}\] (here, note that $\partial_v(u_i)$ is ought to vanish on the same component $u_i$ does due to $v$ being a derivation, so $z_i$ are well defined on $\Sing(U)$). So, \[L_{v}(\alpha^*) = -(\underset{i=1}{\overset{n}{\sum}} \frac{\partial_v u_i}{u_i})\alpha^*\]

This map is clearly linear w.r.t. the multiplication of $v$ by functions, which means that $\ad_{\xi}$ is linear.
\end{proof}

Provided $X$ has a dimension $2$ and assumption \ref{as:local-unobstructed} holds, we have \newline $H^{*}(T_X(\log))$ (non-canonically) equal to the cohomology of $\ad_{\xi}$ acting on $H^*(T_X^0 \oplus T_X^1[1])$ (higher terms of spectral sequence vanish for dimensional reasons).

\begin{note} There are actually two approaches to the smoothing. Initial approach, due to Friedman, was only about dimension two, and assumption \ref{as:local-unobstructed} held trivially, so it was mainly concerned about the action of $\ad_{\xi}$ on $H^*(T_X^0 \oplus T_X^1[1])$. Latter approach, due to Kawamata and Namikawa worked in arbitrary dimension and was mostly concerned about logarithmic cohomology, but used Calabi-Yau condition, which allowed them to smooth out varieties of higher dimension. This approach was later generalized to (generalized) Fano varieties by Tziolas \cite{Tz}. We are neither in Fano, nor Calabi-Yau case - so we choose the former, earlier approach to analyze our problem.
\end{note}

From now on, we assume the dimension $2$ and assumption \ref{as:local-unobstructed}.

\begin{theorem} \label{thm:log-def} Suppose \[\ad_{\xi}: H^1(T_X^0) \rightarrow H^1(T_X^1) = H^2(T_X^1[1])\] is surjective (or, equivalently, $H^2(T_X(\log)) = 0$) Then, there are smoothing deformations of $X$. \end{theorem}

This is a Corollary 2.4. from \cite{KN}. We now give the geometric self-contained proof for the reader's convenience.

\begin{proof}Consider the quadratic map \[K_2: H^1(T_X) \rightarrow H^2(T_X)\] Choose the non-canonical splitting $H^1(T_X) = H^1(T_X^0) \oplus (\mathbb{C} = H^0(T_X^1))$. Denote the coordinates on $H^1(T_X^0)$ as $x_1, ..., x_d$ and the coordinate along $H^0(T_X^1)$ as $t$. Then, \[K_2(x_1, ..., x_n, t) = (\sum_i \ell_i x_i)t + r t^2, \ell_i, r \in H^1(T_X^1)\] - there are no quadratic terms in $x$ due to $[\cdot,\cdot]$ being in agreement with filtration. Moreover, the linear map $\sum_i \ell_i x_i: H^1(T_X^0) \rightarrow H^1(T_X^1)$ is precisely $\ad_{\xi}$. Now, note that one can force $r=0$ by the linear change of coordinates, due to our assumption that $\ad_{\xi}$ is surjective. Assume now w.l.o.g that it is the case.

Now, note that the whole series $K$ is actually divisble by $t$, either by noting that Massey products are also in agreement with the filtration or noting that the locus $t=0$ is a locus of locally-trivial deformations, which are unobstructed by assumption \ref{as:local-unobstructed} and so it lies in the formal scheme $K=0$.

Then, the series $K/t = \ad_{\xi} + ...$ start with the linear term which is a surjective linear mapping, so define a smooth formal subvariety by the formal implicit function theorem. This subvariety is also clearly transverse to the locus of locally trivial deformations $t=0$. This implies existence of the formal smoothing deformation, and to make it non-formal apply Artin's approximation theorem (see [KN] for details).
\end{proof}

Our main concern for the rest of this section is developing methods to calculate $H^*(T_X(\log))$. Our approach is geometric in nature - we interpret $T_X(\log)$ as a sheaf controlling deformation theory of a certain object, and then deem to consider the corresponding moduli space by other, geometric, means.

\begin{theorem}[Friedman, {\cite{Fr}}, Thm. 4.5] \label{thm:txlog}
$H^1(T_X(\log))$ is a tangent space to the locally trivial deformations which preserve $d$-semistability.
\end{theorem}
\begin{proof}
Consider the covering by log atlas $(U^1, ..., U^s)$, with corresponding structure functions, $\alpha$ - corresponding nonvanishing section of $(T_X^1)^*$, $v^{ij} \in T_X (U^i \cap U^j)$ - Cech cocycle. The Cech cocycle deforming $T_X^1$ is then $\varphi^{ij} = L_{v^{ij}} (\alpha) / \alpha$, which is trivial iff $[v, \xi] = 0$
\end{proof}

\begin{assumption}
From now on, $\Sing(X)$ only has rational irreducible components.
\end{assumption}

\begin{definition}
Denote as $M_X$ the space of locally trivial deformations of $X$. Denote as $J_X$ the space of line bundles on $\Sing(X)$ with $c_1 = 0$. Denote by $O: M_X \rightarrow J_X$ the map which sends $X$ to $T_X^1$.
\end{definition}

\begin{note}
While $\Sing(X)$ could indeed deform, $J_X$ is always $(\mathbb{C}^*)^{h^1(\Sing(X))}$, identified canonically. See Lemma \ref{lem:glue} for details.
\end{note}

\begin{corollary} \label{cor:moduli-smooth}
If the fiber $O^{-1}(\mathcal{O}_{\Sing(X)})$ is smooth in $X$, there exists a smoothing of $X$.
\end{corollary}

\begin{proof}
Inspecting the proof of the Theorem \ref{thm:txlog} we see that $\ad_{\xi}$ is a differential of the map $O$ in the point $X$. Fiber being smooth implies surjectivity of this differential, which, in turn, implies existence of smoothing by Theorem \ref{thm:log-def}.
\end{proof}

Our approach to calculation of $H^*(T_X(\log))$ is now the following: describe $M_X$ as explicitly as we can in terms of gluing of components and calculate $O$.

Summarising this section: results above follow easily from the results of Friedman and Kawamata-Namikawa, Corollary \ref{cor:moduli-smooth} is used (in some form) by Friedman. The references are for SNC case, however actual proofs do work for NC case just as well - at the cost of a bit more confusing notation. We tried, however, to present our own point of view, based largely on geometric consideration of the Kuranishi mapping. We conclude the section with additional theorem, an improvement of Corollary \ref{cor:moduli-smooth} which seems to be not known in the log deformation theory. This fact is not needed in the current work but might find its applications later.

\begin{theorem}\label{thm:cool-smooth}
Suppose $O^{-1}(\mathcal{O}_{\Sing(X)})$ has the dimension $\dim(M_X)-\dim(J_X) = h^1(T_X^0) - h^1(T_X^1)$, yet not necessary smooth or even reduced. There still exists the surface $X'$ with $\Delta_{X'}$ homeomorphic to $\Delta_X$, admitting the smoothing deformation.
\end{theorem}

Note, that in standard deformation theory, it is enough to prove that $h^1(T_X) > h^2(T_X)$ to guarantee the existence of deformations. In the logarithmic deformation theory, though, it is not the case. This is, basically, our version of obstructed deformation theory. Curiously, the proof uses almost no log deformations and is mainly about the geometry of the Kuranishi space.

\begin{proof} Assume the notation from the proof of Thm. \ref{thm:txlog}, i.e., choose the non-canonical splitting of $H^1(T_X^0) \rightarrow H^1(T_X) \rightarrow H^0(T_X^1) = \mathbb{C}$. Pick the coordinate $t$ to be the pullback of the standard coordinate on $H^0(T_X^1)$ dual to $\xi$, and $(x_1, ..., x_d)$ - the rest of coordinates, defining splitting. We now fix our attention to an open neighborhood of the zero in which Kuranishi mapping converges.

Denote the analytic subset $K=0$ as $Z$, and the hyperplane $t=0$, which is an image of $H^1(T_X^0)$ as $H$.

It is easy to see that $H \subset Z$ - that is true because we have demanded in Assumption \ref{as:local-unobstructed} that locally trivial deformations are unobstructed, and $H^1(T_X^0)$ is the space of local deformations.

That implies that $t|K$ - because $K$ is the defining equation of $Z$, and $t$ is prime. Denote as $Z'$ the analytic subset $K/t = 0$.

Now, the set $Z' \cap H$ is the set of locally trivial deformations that preserve $d$-semistability. We have done it in the smooth case by calculation in Theorem \ref{thm:txlog}, but it can be proved geometrically, too. $Z' \cap H$ is clearly the same as $\Sing(Z) \cap H$, i.e., it is the subscheme of such points $p \in H$ that have tangent space bigger than $T_pH$. However, being $d$-semistable exactly means having $H^1(T_X) \neq H^1(T_X^0)$.
\begin{center}
\includegraphics[clip, width=0.7\linewidth]{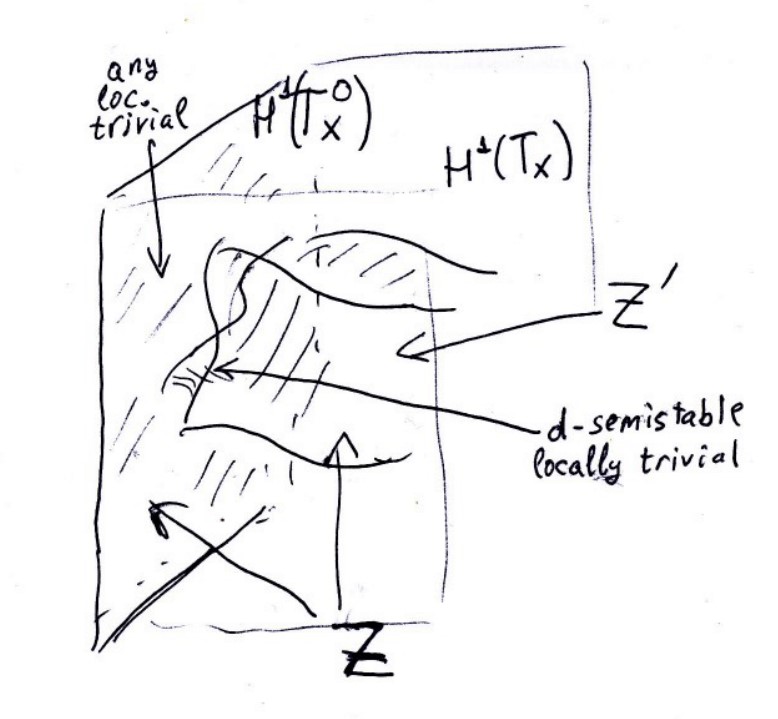}
\end{center}
Now, we proceed with dimension counting. Denote $h^1(T_X^0) = d$, $h^1(T_X) = d+1$, $h^2(T_X) = h^1(T_X^1) = r$.

$Z'$, being a fiber of the map $K/t: H^1(T_X) \rightarrow H^2(T_X)$, has $\dim(Z') \geq d+1-r$. By assumption of the theorem, $\dim(Z' \cap H) = d-r$. It implies that $Z' \setminus H$ is non-empty, moreover, we have an arc $\gamma: D \rightarrow Z'$ which is not contained in $H$ and $\gamma(0) = X$. It also implies $\dim(Z') = d+1-r$. Denote the coordinate on a disk $\tau$.

This $\gamma$ corresponds to some flat family $\tilde{X}$ with central fiber $X$, which, however, does not necessarily have the smooth total space.

Suppose $\gamma$ has tangency $k$ with $H$. Then, we would like to consider $k$-th Kodaira-Spencer differential $\partial_{\tau}^k (t \circ \gamma) (d\tau)^k = \lambda \xi (d\tau)^k$. Let's see that it is defined correctly.

Explicitly, consider the $(k-1)$-st order of deformation \[ \tilde{X}^{(k-1)} = \tilde{X} \underset{D}{\times} \Spec(\mathbb{C}[\tau]/\tau^k)\] It is still locally trivial at this stage, so one can promote $\alpha^*$ to it (it is a choice, we choose to do it in such a way that $\xi$ is dual to $dt$ in $H^1(T_X')$ for any locally trivial deformation $X'$, and our choice of $t$ was arbitrary). Then, possible liftings of $\tilde{X}^{(k-1)}$ to the family over $\mathbb{C}[\tau]/\tau^{k+1}$ are a torsor over $H^1(T_{\tilde{X}^{(k-1)}})$, and this torsor has a distinguished codimension $1$ subspace - the subspace of locally trivial liftings. The quotient by this subspace is, then, canonically $H^0(T_{\tilde{X}^(k-1)}^1) = \mathbb{C}$, which provides us with the desired mapping - so, as expected, $k$-th order Kodaira-Spencer derivative is not canonical, yet its projection which measures the deviation from the locally trivial deformation is.

Now, consider the same family in the neighborhood of a singular point on $X$, denote its germ as $\mathcal{X}$. The deformations of the germ are controlled by $\Gamma(T_{\mathcal{X}}^1)$, and indeed the deformation is locally trivial up to the $k$-th order, so it has $k$-th Kodaira-Spencer differential, which is equal to $\lambda \alpha^*(d\tau)^k$.

That means that the deformation locally has the form $x_1 ... x_s = \tau^k$, here $s = 2, 3$, which implies, first, that the generic fiber is indeed smooth, and, second, that the singularities of the family are toroidal. By the results of \cite{dFKX} we can resolve them and apply semistable reduction theorem without changing the homeomorphism type of the dual complex.
\end{proof}

\section{Example}

We present our example. Let us recall that we are interested in constructing smoothable geometrically maximally degenerate normal crossing surface $X$ with non-contractible, collapsible dual complex. We decided to use the simplest example of such complex, topological duncehat (see Def.\ref{def:dunce}).

That means that $X$ is some surface with one irreducible component, self-intersecting in one curve, and having one triple point. Consider the normalization $\nu:\hat{X}\rightarrow X$ and the gluing locus $\nu^{-1}(\Sing(X))$. Due to the absence of branch switching (which is due to $\Sing(X)$ having rational normalization) the gluing locus is a union of two distinct rational curves, which we denote $P, Q \subset \hat{X}$. Their normalizations are denoted $\hat{P}, \hat{Q} \overset{\nu_P, \nu_Q}{\longrightarrow} P, Q$ respectively. These curves only have nodal intersections and self intersections and are glued by some mapping $\varphi: \hat{P} \rightarrow \hat{Q}$. It is also clear that the gluing locus $P \cup Q$ should have three nodal points in total - as they glue by three into one triple point on $X$.

That leaves two possibilities for the combinatorial structure of $X$:

\begin{itemize}
    \item \textbf{Wrong case} $P$ and $Q$ are two smooth curves intersecting each other in three nodes.
    \item \textbf{Right case} $P$ and $Q$ both have one node and intersect each other once.
\end{itemize}

It turns out that the first case the dual complex is actually the different triangulated space, the triangle $[012]$ with edges $[01], [12]$ and $[20]$ glued together. It has a fundamental group $\mathbb{Z}/3\mathbb{Z}$, and is not too interesting to us (while our program can, indeed, be carried out for it).

The second case, however, is what we are interested in. Denote as $p_N, q_N$ the nodes of $P, Q$, as $n$ the intersection of $P$ and $Q$, as $p_1, p_2$ - preimages of $p_N$ on $\hat{P}$ (ordered arbitrarily), as $q_1, q_2$ - preimages of $q_N$ on $\hat{Q}$, as $p_3, q_3$ - preimages of $n$ on $\hat{P}, \hat{Q}$ - respectively,

Now, recall that $\hat{P}, \hat{Q}$ are genus $0$, so the map $\varphi$ can be defined by images of three points. We now define \[\varphi(p_1) = q_2\]
\begin{equation}\label{eq:phi-def} \varphi(p_2) = q_3\end{equation}\[\varphi(p_3) = q_1\]

Let us explain why this gluing indeed has $\Delta_X$ as its dual complex. Let us denote $b_{P;1}, b_{P;2}$ the branches of $P$ passing through $p_N$ which have $p_1, p_2$ respectively on their lift to $\hat{P}$, and do similarly for $b_{Q;1}, b_{Q;2}$. Denote as $b_{P; n}$, $b_{Q; n}$ the branches of $P$ and $Q$ passing through $n$.

\newpage

Then, it is readily seen that the neighborhoods of $p_N, q_N, n$ glue into the neighborhood of the coordinate cross in $\mathbb{C}^3$ - $b_{P;1}$ glues to $b_{Q;2}$ which passes through $b_{Q;1}$ which glues to $b_{P;n}$ which passes through $b_{Q;n}$ which glues to $b_{P;2}$ which passes through $b_{P;1}$ (see picture below).

\begin{center}
\includegraphics[clip, width=0.5\linewidth]{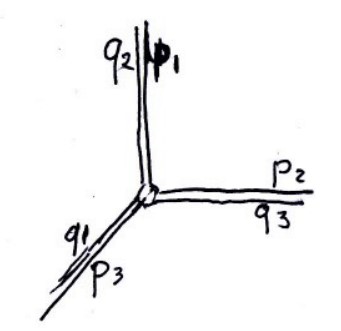}
\end{center}

Now, the edges of the triangle in $\Delta_X$ correspond to the branches of $\Sing(X)$ passing through the triple point (namely, $b_{P;1} = b_{Q;2}, b_{Q;1} = b_{P;n}$, $b_{Q;n} = b_{P;2}$). 

Vertices of the triangle correspond to the branches of $\hat{X}$ passing through the triple point, let us denote $[0]$ the branch passing through $p_N$, $[1]$ - passing through $q_N$, $[2]$ - branch passing through $n$.

The edge orientation space is the ordering of branches of $\hat{X}$ passing through the curve corresponding to the edge. Then, set that the edges are oriented from the branch passing through $P$ to the one passing through $Q$. This convention gives the following orientation of the edges: $[01], [21], [02]$. Up to the permutation, it gives the gluing of the duncehat complex. (see picture below)

\begin{center}
\includegraphics[clip, width=0.5\linewidth]{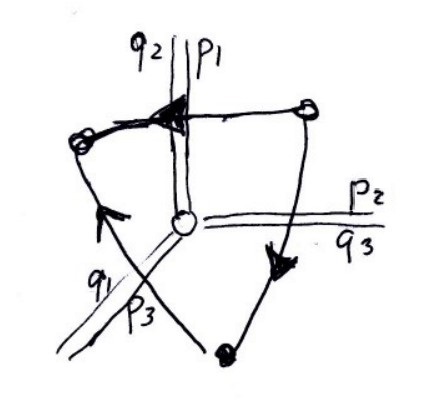}
\end{center}

Now, we would like to find appropriate surface and curves on it with combinatorics described above. Another thing to look for is Kulikov's triple point formula, which states that

\begin{proposition}
For $\nu_C: C \rightarrow \Sing(X)$ any normalization of the irreducible component of $\Sing(X)$ we have:

\[c_1(\nu_C^*T_X^1) = c_1(N(C; B_1)) + c_1(N(C;B_2)) + \tau\]

where $B_1, B_2$ - branches of $X$ passing through $C$, $\tau$ - amount of triple points on $C$.
\end{proposition}

Kulikov's condition then states that aforementioned quantity is zero, and is a topological requirement for $d$-semistability (see Lemma \ref{lem:Kul-c1})

\begin{definition}
Consider a projective plane $\mathbb{P}^2$ with a pair of nodal cubics (denote them $P'$, $Q'$). Suppose these cubics do not pass through each other's nodes. Blow up $\mathbb{P}^2$ in $8$ of $9$ of their intersections (if $P', Q'$ are not tranverse, we blow up multiple intersections multiple times). Additionally, blow up any smooth point on the proper preimage of $P'$ except the ninth, remaining intersection. Denote the resulting blow up as $\hat{X}$, the proper preimages of $P'$ and $Q'$ as $P$ and $Q$. Make the choices of $p_1, p_2$ and $q_1, q_2$. The data above is called the \textbf{Main Construct}.
\end{definition}

The rest of proof is concerned with the space of moduli of these constructs and the mapping $O$.

\begin{proposition} Kulikov's triple point condition is satisfied for the main construct.
\end{proposition}

\begin{proof}
The (proper) self-intersection of the nodal cubic is $7$. Every blowup lowers both self-intersections by $1$, and one additional blowup makes the self intersection of $P$ to be $-2$, and self-intersection of $Q$ to be $-1$. Adding the amount of triple points on the normalization of $\Sing(X)$, which is $3$, we obtain $0$, as predicted by Kulikov's condition.
\end{proof}

\section{Obstruction mapping}

We now would like to understand the obstruction mapping. We need a few technical preparations.

\subsection{Setup}
At first, let us state precisely what are the spaces $M_X$ and $J_X$ are.

\begin{definition}
$M_X$ is the moduli space of constructs, i.e. moduli space of pairs of nodal cubics $P$, $Q$ on $\mathbb{P}^2$ (we have called them $P'$, $Q'$ in previous section, but drop the superscript from now on as it is unlikely to cause confusion).

These cubics are subject to following condiitons:

\begin{itemize}
    \item $P$ and $Q$ are non-degenerate nodal cubics - so, they do not develop cusp or do not become reducible.

\item $P$ doesn't pass through the node of $Q$ and vice versa.

\end{itemize}

Also, these cubics are subject to following additional choices:

\begin{itemize}

\item $\nu_P^{-1}p_N$ and $\nu_Q^{-1}q_N$ are ordered, them in order are denoted $(p_1, p_2)$, $(q_1, q_2)$.

\item An intersection point $n \in P \cap Q$ is chosen.

\item An additional point $b \in P$, satisfying $b \neq n, p_N$, is chosen.
\end{itemize}

These objects are considered up to the projective plane transformations.
\end{definition}

\begin{proposition}
$\dim(M_X) = 9$
\end{proposition}

\begin{proof}
Each cubic has $8$ parameters. Projective transformations also subtract $8$, and one additional is for an additional point $b$.
\end{proof}

\begin{definition}
$J_X$ is the space of line bundles with $c_1 = 0$ on \newline $C = \mathbb{P}^1/(0 \sim 1 \sim \infty)$.
\end{definition}

\begin{proposition}
$\dim(J_X) = 2$. It is canonically equivalent to the two-dimensional torus $(\mathbb{C}^*)^3 / (\mathbb{C}^*)$ (quotient by the subgroup generated by $(1;1;1)$).
\end{proposition}

\begin{proof}
Consider the line bundle $L$ over $C$. $\nu_C^{*}L$ is a trivial bundle on $\mathbb{P}^1$, consider its trivializing section $\kappa$. Then, $(\kappa(0), \kappa(1), \kappa(\infty))$ is a defined up to the multiplication by constant triple of elements in the $1$-dimensional space $L(0 \sim 1 \sim \infty)$. It gives the desired element in $(\mathbb{C}^*)^3 / (\mathbb{C}^*)$.

The fact that the described mapping is an isomorphism is also fairly obvious (see lemma \ref{lem:glue} further for details).
\end{proof}

\subsection{Line bundles on multicomponent curves}
Suppose $R$ is a curve with triple points only, and with rational components only. Denote by $\mathcal{S}=(R_1, ..., R_k)$ the set of normalizations of components of $\nu_i: R_i \rightarrow R$, and by $\mathcal{T} = (t_1, ..., t_s)$ the set of triple points.

Then, there is a bipartite graph, $\mathcal{S} \overset{\Gamma}{\leftrightarrow} \mathcal{T}$, connecting each component with triple points on it. We allow multiple edges - in case a component passes through a triple point multiple times.

For each point $t_i \in \mathcal{T}$ form the group $G_i$ = $(\mathbb{C}^*)^3$, with factors of the product labeled by edges of $\Gamma$ coming from $t$, and denote by $\delta_i: \mathbb{C}^* \rightarrow G_i$ the embedding of the diagonal group (generated by $(1;1;1)$). Denote $G_i/\im(\delta_i) = \tilde{G}_i$

For each component $R_i$ form the group $H_i = \mathbb{C}^*$.

There is a natural map $q: \prod H_i \rightarrow \prod G_j$ which maps each $H_i$ to all $G_j$'s such that $(ij)$ is an edge of $\Gamma$, to the component(s) of $G_j$ corresponding to the edge(s) $(ij)$.

\begin{lemma}[gluing lemma]\label{lem:glue}
The space of line bundles with $c_1 = 0$ on $R$ is canonically $\prod \tilde{G}_i / \im(q)$. The group law is also in accord with this description. We will refer to the elements of this group as \textbf{gluing data}.
\end{lemma}

\begin{proof}
Suppose $L$ is a line bundle over $R$, and pick $\gamma_1, ..., \gamma_k$ - sections of $L_i = \nu^*_i(L)$. Then, in every fiber of $L$ over a triple point $t_j$ we obtain values of three sections - which gives an element in $\tilde{G}_j$. The choice of $\gamma_i$'s was up to a constant, which gives the quotient by $q$.

Moreover, it is clear that any line bundle can be obtained in this way - given an element in $\prod G_j$ we can glue the fibers of $L_i$ together using this data in an obvious way, the group law check is also straightforward.
\end{proof}

\subsection{Explicit description of $T_X^1$}

\begin{lemma} \label{lem:tx1-explicit}
Suppose $X$ is a normal crossing surface without branch switching, $\nu: C \rightarrow \Sing(X)$ - normalization of some irreducible component of the singular locus, $t_1, ..., t_k \in C$ - preimages of the triple points, $B_1, B_2$ - branches of $X$ passing through $C$.

Then, $\nu^*(T_X^1) = N_{C; B_1} \otimes N_{C; B_2} ([t_1] + ... + [t_k])$.

Suppose $t \in X$ - a triple point, $C_1, C_2, C_3$ - branches of $\Sing(X)$ passing through $t$, $B_{ij}$ - branches of $X$ in $t$ passing through $C_i, C_j$.

Then, the fiber of $N_{C_1; B_{12}} \otimes N_{C_1; B_{13}} ([t])$ in $t$ is canonically identified with $T_{t; C_1} \otimes T_{t; C_2} \otimes T_{t;C_3}$, and this identification gives the gluing of $T_X^1$ between the components.

\end{lemma}

\begin{proof}
It is easier to work with the dual bundle, $(T_X^1)^* = \Hom(T_X^1, \mathcal{O}_{\Sing(X)})$. For the log chart $(u_1, u_2, u_3)$ near the triple point $t$ we can form the object $\omega_1 = u_1 du_2 \otimes du_3$. Provided $C_1$ is a component on which $u_2 = u_3 = 0$, $\omega_1$ is a section of $N^*_{C_1; B_{12}} \otimes N^*_{C_1; B_{13}}([-t])$. It easy to check that equivalent log charts give the same $\omega_1$.

Now, in the same way log structure gives the sections $\omega_2$, $\omega_3$ on the branches $C_2$, $C_3$, which gives the gluing between fibers in $t$ - $N^*_{C_1; B_{12}} \simeq T_{C_2}^*$, $N^*_{C_1, B_{13}} \simeq T_{C_3}^*$, $\mathcal{O}_{C_1}(-[t]) \simeq T_{C_1}^*$. That identifies the fiber of $(T_X^1)^*$ in $t$ with $T^*_{C_1} \otimes T^*_{C_2} \otimes T^*_{C_3}$ canonically.
\end{proof}

\begin{corollary}\label{cor:tx1-explicit}
Given $X$ geometrically maximally degenerate, satisfying Kulikov's triple point condition, pick for every component $\nu_i: C^i \rightarrow \Sing(X)$ with branches $B^i_1$, $B^i_2$ the pair of conormal forms $\omega^i_1 \in N^*_{C^i, B^i_1}, \omega^j_2 \in N^*_{C^i, B^i_2}$ such that $\omega^i_1 \otimes \omega^i_2$ vanishes in all triple points on $C^i$. 

Then, gluing data is obtained as differential of $\omega^i_1 \otimes \omega^i_2$ along $C^i$.
\end{corollary}

\subsection{Calculation in the example} Before the start of the computation, we must encourage the reader. The computation which will shortly follow looks very tedious on first sight, however most of the complexity comes from the fact that we need to somehow trivialize $1$-dimensional tangent spaces to the points $(p,q)_{(1,2,3)}$. There are different such trivializations, and different tensors look better in different trivializations. It leads to a lot of multiplicative factors depending on how these trivializations relate, and this dependence is contained in two parameters - the positions of flex points of the cubics in the projective coordinates given by points $p_1, p_2, p_3$. We will later promptly ignore such multiplicative factors and obtain a nice answer modulo them - and it will be enough to prove the desired result. 
So, on the promise of the final answer being nice modulo factors we can ignore, we proceed with calculation of the gluing data of the main construct:

We assume the following convention: while calculations will take place on $\mathbb{P}^2$, the calculation of all divisors will take place on $\hat{X}$ directly.

Denote as $\Omega_P, \Omega_Q$ the sections of $\omega_{\mathbb{P}^2}$ such that $\frac{1}{\Omega_P}$ vanishes in $P$ and similarly for the other one. They are chosen uniquely up to a constant - which we will fix in a specific way.

Denote as $\alpha_P$, $\alpha_Q$ Poincare residues of these forms - these are $1$-forms on $\hat{P}$ and $\hat{Q}$, with residues in $p_1, p_2$ and $q_1, q_2$ respectively.

We choose $\Omega_P$ and $\Omega_Q$ in such a way that the $\Res_{p_1}(\alpha_P) = \Res_{q_1}(\alpha_Q) = 1$.

Denote as $\tau_p$ the projective invariant coordinate on $\hat{P}$ such that \[\tau_P(p_1) = 0\]\begin{equation} \tau_P(p_2) = \infty\end{equation} \[\tau_P(\psi_P) = 1\] where $\psi_P$ is a flex point of $P$. Do the same for $Q$.

Denote \begin{equation}n_P = \tau_P(p_3), n_Q = \tau_Q (q_3)\end{equation}

Pick some arbitrary vector fields $v_P, v_Q$ defined only in terms of these coordinates. Specifically, we use
\begin{equation}v_P = (\tau_P-1)^2 \frac{\partial}{\partial \tau_P}\end{equation} and the same for $Q$. The only requirement we have is that these fields are nonzero in $(p,q)_{1,2,3}$ provided position is generic enough.

Pick also some arbitrary affine chart, generic enough, and denote its infinity as $H$, and its volume form as $\Omega$. Define affine equations of $P$ and $Q$ as follows:
\begin{equation}
    f_{P,Q} = \frac{\Omega}{\Omega_{P,Q}}
\end{equation}

We would also like to pick some functions which vanishes in our points of interest. We suggest following: \begin{equation}s = n_P^2\frac{\tau_P - n_P}{(\tau_P - 2n_P)^3}\end{equation}
This function is chosen in such a way that it has the same formula in the coordinate defined by $p_1, p_2, p_3$, independently on $n_P$.

Finally, we pick the sections of the conormal bundle to $P$ and $Q$:

\begin{equation}\label{eq:beta}
    \beta_P = i_{v_P} \Omega, \beta_Q = i_{v_Q} \Omega
\end{equation}

From now on, we identify the curve $\hat{Q}$ with $\hat{P}$ via the gluing map $\varphi$. Finally, we consider the following meromorphic section of $T_X^1$ (pulling back objects from $Q$ via the gluing map): \begin{equation}\eta = s \beta_P \otimes \beta_Q \end{equation}

\begin{proposition}
$\eta$ vanishes once in $p_1$, $p_2$, $p_3$ for $P, Q$ in generic position.
\end{proposition}

\begin{proof}Direct check.\end{proof}

Now, denote \[\lambda_2 = \frac{v_P(p_1)}{\varphi^*(v_Q (q_2))}\]\begin{equation}
\lambda_3 = \frac{v_P(p_2)}{\varphi^*(v_Q (q_3))}   
\end{equation}\[\lambda_1 = \frac{v_P(p_3)}{\varphi^*(v_Q (q_1))}\]

\begin{proposition}
The gluing data for the section $\eta$ is as follows:

(along the branch $p_1$)
\begin{equation} \label{eq:main1}
     \frac{\partial s}{\partial v_P}(p_1) \times \langle \beta_P (p_1), v_P(p_2)\rangle \times \langle\beta_Q (q_2), \lambda_1 v_Q(q_1)\rangle = 
\end{equation}
\[   \lambda_1 \frac{\partial s}{\partial v_P}(p_1) \times  \Omega(v_P(p_1), v_P(p_2)) \times \Omega(v_Q(q_2), v_Q(q_1)) 
\]

(along the branch $p_2$)
\begin{equation}\label{eq:main2}
    \frac{\partial s}{\partial v_P}(p_2) \times \langle \beta_P (p_2), v_P(p_1)\rangle \times \langle\beta_Q (q_3), v_P(p_3)\rangle = 
\end{equation}
\[\frac{\partial s}{\partial v_P}(p_2) \times \Omega(v_P(p_2), v_P(p_1)) \times \Omega(v_Q(q_3), v_P(p_3))
\]

(along the branch $p_3$)
\begin{equation}\label{eq:main3}
     \frac{\partial s}{\partial v_P}(p_3) \times \langle \beta_P(p_3), \lambda_3 v_Q(q_3) \rangle \times \langle \beta_Q(q_1), \lambda_2 v_Q(q_2) \rangle =
\end{equation}

\[\lambda_2 \lambda_3 \frac{\partial s}{\partial v_P}(p_3) \times \Omega(v_P(p_3), v_Q(q_3)) \times \Omega(v_Q(q_1), v_Q(q_2))
\]

\end{proposition}

\begin{proof}
We use the vector field $v_P$ to trivialize the tangent spaces $T_{p_1}, T_{p_2}, T_{p_3}$ (remark that $\lambda_{i+1} v_Q(q_{i+1})$ is the same vector as $v_P(p_i)$). Then, we just calculate the value of the section in this trivialization. Finally, we use the definition of $\beta$'s, see eq.\ref{eq:beta}.
\end{proof}

\begin{lemma} \label{lem:n_p-dependance}
$\lambda$'s and $\frac{\partial s}{\partial v_P}(p_i)$'s only depend on the values of $n_P$ and $n_Q$.
\end{lemma}

\begin{proof}
For $\frac{\partial s}{\partial v_P}(p_i)$ it is evident from the definitions. For $\lambda$'s, the reason is following - pick the coordinate $n_P^{-1} \tau_P$, i.e., such that $p_3 = 1$, $p_1 = 0$ $p_2 = \infty$. The positions of flex points of $P$ and $Q$ are functions of $n_P$ and $n_Q$, so the fields $v_P$ and $v_Q$ in this coordinate depend only on the aforementioned entities.
\end{proof}

Now, this gluing data is by no means final, because the section $\eta$ was actually only rational - it had zeroes and poles away from the set $p_1, p_2, p_3$. Now we need to find a function $r$ such that $r\eta$ has all these zeroes and poles cancelled - and multiply the gluing data obtained by $(r(p_1), r(p_2), r(p_3))$.

Split the divisor of the function $s$ into positive and negative parts:
\begin{equation}D(s) = D(s)_+ + D(s)_-\end{equation}
\[D(s)_+ = [p_1] + [p_2] + [p_3]\]
Perform the calculation of the divisor $D(\eta) - D(s)_+$:

\begin{equation}
    D(\eta) - D(s)_+ = D(s)_- + D(\beta_P) + D(\varphi^*\beta_Q)
\end{equation}

\begin{equation}
    D(\beta_P) = D(v_P) + D(\Omega|_P) = 
\end{equation}
\[D(v_P) - 3[H|_P] + ([b] + [Q|_P] - [p_3])\]

Here, in brackets we put part of the divisor which occurs due to the change of $\Omega$ under blow-up. Similarly, for $Q$:

\begin{equation}
    D(\beta_Q) = D(v_Q) + D(\Omega|_Q)) =
\end{equation}
\[D(v_Q) - 3[H|_Q] + ([P|_Q] - [q_3])\]

Denote \begin{equation}s_b =\frac{\tau_P - 1}{\tau_P - \tau_P(b)}\end{equation}

\begin{lemma} \label{lem:eta'-independent}
The divisor of the section \[\eta' = (f^{-1}_Q)|_P \cdot \varphi^*( (f^{-1}_P)|_Q) \cdot s_b \cdot \eta\]
is dependent only on $n_P, n_Q$
\end{lemma}

\begin{proof}
From the calculations above this divisor is \[D(\eta') = D(v_P) + \varphi^{-1}(D(v_Q)) - [p_3] - [p_2] + D(s) + [\psi]\]

$D(v_P)$ is immobile, $\varphi^{-1}(D(v_Q))$ depends only on relative position of flex points, i.e. on $n_P$ and $n_Q$, $D(s)$ is immobile.
\end{proof}

Pick a function $g = \frac{\tau_P - n_P}{(\tau_P - 1)^2}$. It is dependent only on $n_P$ and has a divisor \begin{equation}D(g) = [p_2] + [p_3] - 2[\psi_P]\end{equation}

\begin{corollary} \label{cor:geta'}
The section $g \eta'$ has a divisor dependent only on $n_P$, and vanishes exactly once in $p_{1,2,3}$
\end{corollary}

Now, finally, we can pick a function $h$ which will cancel this additional part - it will depend only on $n_P, n_Q$ and be defined up to a constant. In what follows, we will not care about the dependance of the gluing data on $n_P$ and $n_Q$, so there is no point in finding this function precisely.

To finalize the answer, we would like to calculate 
\[g \cdot \frac{1}{f_Q|_{P}} \cdot \varphi^*( \frac{1}{f_P|_{Q}}) (p_1)
\]
\begin{equation}
g \cdot \frac{1}{f_Q|_{P}} \cdot \varphi^*( \frac{1}{f_P|_{Q}}) (p_2)
\end{equation}
\[g \cdot \frac{1}{f_Q|_{P}} \cdot \varphi^*( \frac{1}{f_P|_{Q}}) (p_3)
\]

Let us proceed. From now on, we will use the sign ''$\sim$'' to say that two values are related by something dependent only on $n_P$ and $n_Q$.

\begin{equation} \label{eq:ans1} g \cdot \frac{1}{f_Q|_{P}} \cdot \varphi^*( \frac{1}{f_P|_{Q}}) (p_1) = g(p_1) \cdot \frac{1}{f_Q(p_N) f_P(q_N)} \sim \frac{1}{f_Q(p_N) f_P(q_N)}
\end{equation}

It is harder for two other ones:

\begin{equation}\label{eq:ans2}
g \cdot \frac{1}{f_Q|_{P}} \cdot \varphi^*( \frac{1}{f_P|_{Q}}) (p_2) \sim \frac{\alpha_P}{i_{v_Q}\Omega}(n)\cdot \frac{1}{f_Q(p_N)}
\end{equation}

the reason for this is the fact that $\frac{g}{\varphi^{*}(f_P|_{Q})}$ in the point $p_2$ satisfies \[\frac{g}{\varphi^{*}(f_P|_{Q})} \sim 1/\frac{\partial}{\partial v_Q} f_P|_{Q}\] in the point $q_3$, due to $dg \sim v_Q (q_3)$. This chain continues:
\[1/\frac{\partial}{\partial v_Q} f_P|_{Q} = \frac{\alpha_P}{i_{v_Q}\Omega}\]

Similarly,

\begin{equation}\label{eq:ans3}
g \cdot \frac{1}{f_Q|_{P}} \cdot \varphi^*( \frac{1}{f_P|_{Q}}) (p_3) \sim \frac{\alpha_Q}{i_{v_P}\Omega} (n) \cdot \frac{1}{f_P(q_N)}
\end{equation}

Now, to cancel everything and put it into place, let use also use the following obvious formulas:

\begin{equation} \label{eq:subst-invariant-volumes}
    f_Q(p_N) = \frac{\Omega(v_P(p_1), v_P(p_2))}{\Omega_Q (v_P(p_1), v_P(p_2))}
\end{equation}
\[    f_P(q_N) = \frac{\Omega(v_Q(q_1), v_Q(q_2))}{\Omega_P (v_Q(q_1), v_Q(q_2))}
\]

Using equations \ref{eq:main1}, \ref{eq:main2}, \ref{eq:main3}, taking into account corollary \ref{cor:geta'} and equations \ref{eq:ans1}, \ref{eq:ans2}, \ref{eq:ans3}, and using equation \ref{eq:subst-invariant-volumes} we obtain

\begin{theorem} The gluing data of $T_X^1$ is, up to some multiples dependent only on $n_P, n_Q$:

For $p_1$:

\begin{equation}s_b(p_1) \Omega_P(v_P(p_1), v_P(p_2)) \Omega_Q(v_Q(q_1), v_Q(q_2))\end{equation}

For $p_2$:

\[
s_b(p_2)\Omega_Q(v_P(p_1), v_P(p_2)) \Omega(v_Q(q_3), v_P(p_3))\frac{\alpha_P}{i_{v_Q}\Omega}(n) =
\]
\[= s_b(p_2) \Omega_Q(v_P(p_1), v_P(p_2)) (\alpha_P,v_P)(p_3)\]

And the quantity $(\alpha_P, v_P) (p_3)$ is actually $\sim 1$, so

\begin{equation}\sim s_b(p_2) \Omega_Q(v_P(p_1), v_P(p_2))\end{equation}

Similarly, for $p_3$:

\begin{equation}
    s_b(p_3)\Omega_Q(v_Q(q_1), v_Q(q_2))
\end{equation}
\end{theorem}

Denote $1/\Omega_Q(v_Q(q_1), v_Q(q_2)) = a, 1/\Omega_P(v_P(p_1), v_P(p_2)) = b$.

Now, without loss of generality, we can normalize the gluing data in the following way:

\begin{equation}
\const \times (s_b(p_1), s_b(p_2), s_b(p_3)) \times (1; a; b)
\end{equation}

Here, $(s_b(p_1), s_b(p_2), s_b(p_3))$ is the gluing data of the bundle $\mathcal{O}([b]-[\psi_P])$, and $\const$ is some group element dependent only on $n_P, n_Q$ - so, geometrically, on the cross-ratio of $p_1, p_2, p_3, \psi_P$ and $q_1, q_2, q_3, \psi_Q$.

\begin{theorem} Consider the open subset of $M^u_X \subset M_X$ of such constructs that $p_N, q_N, n$ do not lie on the same line. Every fiber of the map $O: M^u_X \rightarrow J_X$ is smooth, and $O$ is surjective (restricted on this subset).
\end{theorem}

Consider a pair of nodal cubics $P$ and $Q$, satisfying the assumption above, pick a generic affine chart, and restrict ourselves with the following subspace of  $M^u_X$:
\begin{definition}
$M_X^{0}$ is a space of such constructs $(P', b')$, $Q'$ that $(P, b)$ is affine equivalent to $(P', b')$ in our chosen chart, and transform conjugating them preserves the point $n$, and the same is true for $Q$ and $Q'$. Denote by $f_{P'}$, $f_{Q'}$ their defining equations
\end{definition}

The values which were dependent only on $n_P, n_Q$ are constant on such subspace.

\begin{proof}
Then, it is easy to see that on $M_X^0$ \[\Omega_Q(v_{p_1}, v_{p_2}) = \const f_Q(p_N)\]
and for $\Omega_P$ symmetrically.

We would like to show that for any pair $P, Q$ there exists a two-dimensional tangent subspace in $M_X^0$ such that the differential of the map $(f_P(q_N), f_Q(p_N))$ is surjective.

It is very clear, because one can pick the affine transform which preserves $n$ and $p_N$, and moves $q_N$ anywhere not in the line $(n, p_N)$ - in particular, moves it transversely to the level set of the function $f_P$. Moving $Q$ with such an affine transform does not change the value $f_Q(p_N)$, yet changes $f_P(q_N)$ in the first order. Similar argument, of course, works for $P$.

That proves the first part of the theorem - for any construct there exists a two-dimensional tangent subspace in $M_X^0$, and hence, in $M_X^u$, which is projects isomorphically to the tangent space of $J_X$. That implies that the fibers are all smooth.

It is not hard to also come up with the family of constructs which maps to $J_X = (\mathbb{C}^*)^2$ surjectively - one should start with any pair $P, Q$ as above, then uses the transforms described above to make the value $f_Q(p_N)$ arbitrary without changing $f_P(q_N)$, then does the same for $f_P(q_N)$.
\end{proof}

\begin{corollary}\label{cor:epicwin}
There exists a smoothable construct.
\end{corollary}

\section{Properties of the smoothing}

In this section, we discuss the topology of the generic fiber of our smoothing family. Let us pick the new convention - the specific family we have constructed in the previous section will be referred as $\tilde{Y}$ to avoid confusion in the general discussions about any family.

Suppose $\tilde{X}$ is any family with normal crossing central fiber $X$ and smooth total space. Denote as $Z_i^0$ the manifold $Z_i \setminus Z_{i+1}$, and consider the following the bundle $T_i \rightarrow Z_i^0$ by $i-1$-dimensional tori: the torus over a point of $Z_i^0$ is a product of spherisation of normal bundles to $Z_i^0$ inside branches of $Z_{i-1}$ passing through it.

\begin{lemma}
The generic fiber $X^\varepsilon$ is homeomorphic to some space $X^{\mathbb{R}}$ representable as the disjoint union of $T_i$'s.
\end{lemma}

\begin{proof}
This fact seems to be well known, however we did not manage to find the exact reference (there is a writeup by W. D. Gillam \cite{WDG} available on the net, but overall this construction seems to be the part of the folklore) so sketch the proof here. The construction is as follows: consider the real oriented blowup of the family $\tilde{X}$ in the central fiber $X$. It can be checked locally that the resulting space is a manifold with corners, admits canonical structure of the manifold with boundary, and is mapped into the real oriented blowup of $D$ in point $0$ with differential of the full rank. Then, by Ehresmann's lemma, the fibers are diffeomorphic, and the fiber over any preimage of $0$ will be of the form described above.
\end{proof}

\begin{corollary}[also well known]
\[\chi(X^{\varepsilon}) = \chi^c (X \setminus \Sing(X))\]
\end{corollary}

Here, $\chi^c$ is an Euler characteristic with compact support.

\begin{proof}
$\chi^c$ is additive on disjoint union and $0$ for all bundles with toral fibers.
\end{proof}

That alows the computation of cohomology of $Y^{\varepsilon}$:

\begin{corollary}
\[\chi(Y^{\varepsilon}) = 11\]
\[h^{1,1}(Y^\varepsilon) = 9\]
\end{corollary}

\begin{proof}
Let us calculate $\chi^c(Y \setminus \Sing(Y))$. Euler characteristic of $\mathbb{P}^2$ is $3$, we add $9$ for nine blow ups, subtract $4$ for removing two rational curves and add $3$ for three nodal points. $3 + 9 - 4 + 3 = 11$.

The second row just follows from $h^{1,0}=h^{2,0}=0$.
\end{proof}

Then, we would like to check that $Y^\varepsilon$ is a minimal surface of general type (the fact that it is of general type follows from it being not rational and geography of Chern numbers, but minimality does not). It also automatically implies that the fibers of the family are not only analytic, but algebraic projective surfaces.

\begin{lemma}\label{lem:general}
$Y^{\varepsilon}$ is a surface of general type.
\end{lemma}

\begin{proof}
By Noether's formula it has $c_1^2 + c_2 = 12\chi(\mathcal{O}_{Y^\varepsilon}) = 12$, and $c_2 = 11$. As we do not know it is minimal, we conclude that it could be a blow-up of some surface with $c_1^2 + c_2 = 12$ and $c_2 \leq 11$. Such surfaces are either rational or of general type by the classification of surfaces, and we know that $Y^{\varepsilon}$ is not rational, because $\Delta_Y$ is not collapsible.
\end{proof}

\begin{theorem}\label{thm:minimal}The generic fiber of $\tilde{Y}$ is minimal.
\end{theorem}

\begin{proof}
As we already know it is of general type, we can prove that there is no $K_{\tilde{Y}/D}$-negative contraction. Indeed, assume that the generic fiber admits some $K_{\tilde{Y}/D}$-negative curve. Then, there will be a limiting curve in $Y$, with one of its components still being $K_Y$-negative. However, note that $\nu_Y^*(K_Y) = K_{\hat{Y}}([P]+[Q])$. $\hat{Y}$ is a projective plane blown up in $9$ points, $8$ of which lie in $P cap Q$, and one lies on $P$. Denote as $\pi: \hat{Y} \rightarrow \mathbb{P}^2$ the projection on the projective plane, as $E_1, ... E_8$ - exceptional curves which do lie in the preimage of $P \cap Q$, $E_b$ - the last exceptional curve. Then, \[K_{\hat{Y}}([P]+[Q]) = \pi^*(\mathcal{O}(3)) - E_1 - ... - E_8\]

Such bundle is nonnegative on any curve - it is a pullback of $-K$ of the plane with $8$ points blown up, and as these points lie in, say $P$, any curve $C$ which passes through them $k$ times will have the degree at least $\lceil k/3 \rceil$.
\end{proof}

Now, we would like to calculate the fundamental group of $Y^{\varepsilon}$. We shall do it for our combinatorial model $Y^{\mathbb{R}}$.

\begin{lemma}\label{lem:pi1=0}
$\pi_1(Y \setminus \Sing(Y)) = 0$ 
\end{lemma}

\begin{proof}
By Severi's problem \cite{D2}, the fundamental group of the complement of the nodal curve $P \cup Q$ in $\mathbb{P}^2$ is abelian. The blowups can only add relations (because any loop can be perturbed in such a way that it doesn't go through the exceptional locus). Hence, $\pi_1(Y \setminus \Sing(Y)) = H_1(Y \setminus \Sing(Y))$.

Then, $H_1(\mathbb{P}^2 \setminus (P \cup Q))$ is generated by two generators - little loops going around $P$ and $Q$. The blow up in the point $b$ forces the loop around $P$ to become contractible. The blow up in any point of intersection of $P$ and $Q$ makes these loops homotopy equivalent (up to a sign).
\end{proof}

Next theorem we prefer to prove in a bit higher generality, it works for any family $X$ with normal crossing fiber.

\begin{theorem} \label{thm:dxcontractible-pi1}
Suppose $\pi_1(\Delta_X) = 0$, and every component of $X \setminus \Sing(X)$ also has trivial fundamental group. Then, $\pi_1(X^{\mathbb{R}}) = \pi_1(X^{\varepsilon}) = 0$.
\end{theorem}

We split the proof in a few parts. Let us start by perturbing a loop $\gamma$ in such a way that it is transverse to the stratification of $X^{\mathbb{R}}$. Then, the loop will intersect the real codimension $1$ stratum $T_1$ finite amount of times.

\begin{definition}
The \textbf{combinatorial image} of $\gamma$ is a edge-path $\gamma'$ in $\Delta_X$ which moves along the edges corresponding to the components $\gamma$ intersects.
\end{definition}

\begin{lemma}\label{lem:cancel}
In the assumptions of the Theorem \ref{thm:dxcontractible-pi1} every path $\tau$ with combinatorial image $e^{-1}e$ for some edge $e$ (i.e. path going through some codimension $1$ strata and then coming back through it) is homotopic as a path with fixed ends to the path with combinatorial image $1$ (not leaving the open part $X \setminus \Sing(X)$).
\end{lemma}

\begin{proof}
Note that the path $\tau$ enters and leaves through the same codimension $1$ stratum. As it is connected, the path $\tau$ is equivalent to a path of the form: \[\tau = \tau_2 m^{-1} \ell m\tau_1\]
where $\ell$ is a loop in the open part, $\tau_1, \tau_2$ - paths in the open part, $m$ is a short path going through the codimension $1$ component corresponding to $e$. Then, $\ell$ is contractible by the assumption of the theorem, and then $m^{-1}m$ is contracted.
\end{proof}

\begin{lemma}\label{lem:lift}
Any edge-path $\tau'$ in $\Delta_X$ admith a lift - a path $\tau$ in $X^{\mathbb{R}}$ which has the combinatorial image $\tau'$. We can choose the ends of $\tau$ as we wish - provided they lie in the components of $X \setminus \Sing(X)$ corresponding to the ends of $\tau'$
\end{lemma}

\begin{proof}
It is obvious from the fact that the components of $X\setminus\Sing(X)$ corresponding to vertices of $\Delta_X$ are connected - and this is true by the very construction of the dual complex.
\end{proof}

\begin{lemma}\label{lem:lasso}
Any edge-loop $\gamma'$ in $\Delta_X$ is a composition of the elementary loops of the form $(s'_i)^{-1}\sigma'_i(s'_i)$ where $s'_i$ are some paths and $\sigma'_i$ are loops going around exactly one triangle.
\end{lemma}

\begin{proof}
This is true for any contractible loop in any triangulated space and follows easily from the following well-known statement - any contractible edge-loop is contractible by the sequence of elementary operations either swapping an edge to the pair of edges moving around the other side of the triangle, or vice versa. This statement follows from the shellability of any triangulation of $2$-dimensional disk.
\end{proof}

\begin{lemma}\label{lem:global-lift}
Any (transverse to the stratification) loop $\gamma$ in $X^{\mathbb{R}}$ admits the analogous decomposition: \[\gamma = \gamma_n ... \gamma_1 \]
\[\gamma_i = s_i^{-1} \sigma_i s_i\]
with $s_i$ being some paths and $\sigma_i$ being loops with the combinatorial image going around the boundary of $1$ triangle.
\end{lemma}

\begin{proof}
Consider the decomposition of $\gamma'$ and lift $\sigma'_i$'s and $s'_i$'s in some arbitrary way (only thing required is that they are composable). Denote the composition of these lifts $\tilde{\gamma}$. The combinatorial image of $\tilde{\gamma}^{-1}\gamma$ is a path which is removable by operations from Lemma \ref{lem:cancel}, from which it follows that it is contractible.
\end{proof}

\begin{lemma}\label{lem:elementary-move}
For any triangle $m \in \Delta_X$ there exists a loop around the corresponding real codimension $2$ component in $X^{\mathbb{R}}$ which is contractible.
\end{lemma}
\begin{proof}
It is a local check, and the local model is a real oriented blowup of $\mathbb{C}^3$ in the coordinate cross $xyz = 0$. It is easy to see explicitly: the preimage of the point $0$ is a $2$-dimensional torus, and moves from the components $x=0$ to component $y=0$ are unique up to adding the generator $(1;0)$ of its first homology. Similarly, moves from $y =0$ to $z=0$ are unique up to adding another generator, and third move is unique up to adding $(-1;-1)$. That freedom allows to kill any element of the first homology, even in multiple ways.
\end{proof}

Now, we are ready to deduce the theorem.

\begin{proof}[Proof of the Theorem \ref{thm:dxcontractible-pi1}]
By Lemma \ref{lem:elementary-move} we can pick such $\sigma_i$'s in the of the proof of the Lemma \ref{lem:global-lift} that they are contractible.
\end{proof}

\begin{corollary}[of Thm.\ref{thm:dxcontractible-pi1} and Lem.\ref{lem:pi1=0}]

\[\pi_1(Y^{\varepsilon}) = 0\]

\end{corollary}

\section{Questions}

As we have already mentioned, these results are mostly proof of concept - they show that such beasts as families with non-collapsible, contractible dual complex exist, and probably abundant. Yet, the methods and our understanding of how to construct these is still lacking. Identifying example constructed in this way with some known surfaces also doesn't seem too easy.

However, we take the liberty to propose the following 

\begin{conjecture} $Y^{\varepsilon}$ lies in the deformation class of Barlow's surface.
\end{conjecture}

The motivation for this is simply that this is the only known simply-connected surface of general type with such numerical invariants.

The two main roads of development are possible from the current point.

One could be an attempt to construct more examples - and figure out what are the possible constraints of the construction. On this path one will probably encounter some kind of ''tropical'' structure on the dual complex, analogous to the integral flat connections on the sphere arising in the mirror symmetry (see \cite{KS}). After the combinatorial part is settled, one will then need to understand the map $O$ much better than we currently do. This, however, will most likely be less of an issue compared to the first one.

The straightforward way of constructing some examples of interesting dual complexes would be taking a look on the constructions of Korean school \cite{LYP} and then resolving the families obtained by them. We didn't investigate in detail but are planning to.

The second road would be attempting to understand if the dual complex $\Delta_X$ of a geometrically maximal degeneration says anything useful about the deformation class of $X^{\varepsilon}$. Following this road, one would try to perform some elementary equivalences on $\Delta_X$ which would correspond to changes in topology while moving along the boundary of the moduli space of surfaces. Then, this ''simple type'' of the complex will be an invariant of the deformation class of $X^{\varepsilon}$.

The chances of this second program succeeding are rather slim - not only the boundary of the moduli space might not be connected, we can not currently prove the existence of even one geometrically maximal degeneration. We, however, would like to put it as a conjecture - not supported by much but a wishful thinking, but here it is:

\begin{conjecture}
Suppose we are given a \textbf{simply-connected} surface with $h^{1,0} = h^{2,0}$. Then, it admits a geometrically maximal degeneration, moreover, space of such degenerate surfaces is connected.
\end{conjecture}


\newpage
\begin{thebibliography}{50}


\bibitem{B} Rebecca Barlow, \textsl{Some new surfaces with $P_g=0$}, Duke Math. J., Volume 51, Number 4 (1984), 889-904.
\bibitem{D} Deligné, P., Greene, B., \& Yau, S. (1997). Local behavior of Hodge structures at infinity.
\bibitem{D2} Deligné, P., Séminaire Bourbaki 1979/80, Exp.543, pp. 1-10.
\bibitem{dFKX} Tommaso de Fernex, János Kollár, and Chenyang Xu, \textsl{The dual complex of singularities}, Adv. Stud. Pure Math. Higher Dimensional Algebraic Geometry: In honour of Professor Yujiro Kawamata's sixtieth birthday, K. Oguiso, C. Birkar, S. Ishii and S. Takayama, eds. (Tokyo: Mathematical Society of Japan, 2017), 103 - 129
\bibitem{Fr}  Robert Friedman, \textsl{Global Smoothings of Varieties with Normal Crossings},  Annals of Mathematics
Second Series, Vol. 118, No. 1 (Jul., 1983), pp. 75-114
\bibitem{Fu} Takao Fujita, \textsl{On Del Pezzo fibrations over curves}, Osaka Math. J. 27 (1990) 229–245
\bibitem{GS} Mark Gross and Bernd Siebert, \textsl{Mirror Symmetry via Logarithmic Degeneration Data I}, J. Differential Geom. Volume 72, Number 2 (2006), 169-338.
\bibitem{WDG} W. D. Gillam, \textsl{Oriented Real Blowup}, \\ http://www.math.boun.edu.tr/instructors/wdgillam/orb.pdf
\bibitem{Ka} Yasuyuki Kachi, \textsl{Global smoothings of degenerate Del Pezzo surfaces with normal crossings}, Journal of Algebra 307 (2007) 249–253
\bibitem{KN} Yujiro Kawamata, Yoshinori Namikawa, \textsl{Logarithmic deformations of normal crossing varieties and smoothing of degenerate Calabi-Yau varieties}, Inventiones mathematicae, December 1994, Volume 118, Issue 1, pp 395–409
\bibitem{KS} Maxim Kontsevich, Yan Soibelman, \textsl{Affine structures and non-archimedean analytic spaces}, arXiv:math/0406564v1  [math.AG]  28 Jun 2004
\bibitem{Ku} Viktor S. Kulikov, \textsl{Degenerations of K3 surfaces and Enriques surfaces}, Izv. Akad. Nauk SSSR Ser. Mat., 41:5 (1977), 1008–1042; Math. USSR-Izv., 11:5 (1977), 957–989
\bibitem{LYP} Lee, Yongnam \& Park, Jongil. (2006). \textsl{A simply connected surface of general type with $p_g=0$ and $K^2=2$}. Inventiones Mathematicae. 170. 10.1007/s00222-007-0069-7. 
\bibitem{LS} Lev Soukhanov, \textsl{Non-collapsible dual complexes and fake del Pezzo surfaces}, arXiv:1906.10610v3  [math.AG]  10 Jul 2019
\bibitem{P} V. P. Palamodov, \textsl{Deformations of complex spaces}, Uspekhi Mat. Nauk, 31:3(189) (1976), 129–194; Russian Math. Surveys, 31:3 (1976), 129–197
\bibitem{PP} Ulf Persson and Henry Pinkham, \textsl{Degeneration of Surfaces with Trivial Canonical Bundle}, Annals of Mathematics Second Series, Vol. 113, No. 1 (Jan., 1981), pp. 45-66
\bibitem{St} Joseph Steenbrink, \textsl{Limits of Hodge structures} Invent Math (1976) 31: 229. https://doi.org/10.1007/BF01403146
\bibitem{Tz} Nikolaos Tziolas, \textsl{Smoothings of Fano varieties with normal crossing singularities} , arXiv:1005.0531 [math.AG]
\end{thebibliography}
\end{document}